\documentclass[11pt]{article}
\usepackage[T1]{fontenc}
\usepackage{hyperref}
\usepackage{amsmath}
\usepackage{amsthm}
\usepackage{amsfonts}
\usepackage{mathtools}
\usepackage{amssymb}
\usepackage[all]{xy}

\newtheorem{thm}{Theorem}[section]

\newtheorem{lem}[thm]{Lemma}

\theoremstyle{definition}

\theoremstyle{remark}

\newcommand{\R}{\mathbb{R}}
\newcommand{\C}{\mathbb{C}}
\newcommand{\Z}{\mathbb{Z}}
\newcommand{\N}{\mathbb{N}}

\newcommand{\Do}{\mathcal{D}}

\DeclareMathOperator*{\im}{Im}%

\begin{document}

% ----------------------- TITULO, AFILIACIONES, ETC, ---------------------

\title{Toeplitz Operators Acting on True-Poly-Bergman Type Spaces of the 
	Two-Dimensional Siegel Domain: Nilpotent Symbols.}

%----------Author 1
\author%[Hern\'andez-Eliseo]
{Yessica Hern\'andez-Eliseo \\
	Universidad Veracruzana, M\'exico \\
	%}
%\address{%
%	Circuito Gonzalo Aguirre Beltr\'an s/n\\
%	Zona Universitaria \\
%	C.P. 91000 \\
%	Xalapa-Enr\'iquez\\
%	M\'exico}
%\email{
E-mail: yeherel@gmail.com \\ \\
%}
%
%----------Author 2
%\author%[Ram\'irez-Ortega]
%{
	Josu\'e Ram\'irez-Ortega \\ 
	Universidad Veracruzana, M\'exico \\
	%}
%\address{%
%	Circuito Gonzalo Aguirre Beltr\'an s/n\\
%	Zona Universitaria \\
%	C.P. 91000 \\
%	Xalapa-Enr\'iquez\\
%	M\'exico}
%\email{
E-mail: josramirez@uv.mx \\ \\
%}
%
%----------Author 3
%\author%[Hern\'andez-Zamora]
%{
	Francisco G. Hern\'andez-Zamora \\ 
	Universidad Veracruzana, M\'exico \\
%}
%\address{%
%	Circuito Gonzalo Aguirre Beltr\'an s/n\\
%	Zona Universitaria \\
%	C.P. 91000 \\
%	Xalapa-Enr\'iquez\\
%	M\'exico}
%\email{
E-mail: francishernandez@uv.mx
}

%%----------classification, keywords, date
%\subjclass{Primary 30H20, 47L80;  
%	Secondary  %32A25, 32A36, 46L05, 47B38
%	47B35, 42C40}
%
%\keywords{Toeplitz Operator, Poly-Bergman Space, Wavelet Transform}

\date{November 06, 2020}

	\maketitle
	\begin{abstract} 
		We describe certain $C^*$-algebras generated by Toeplitz operators with nilpotent symbols and acting on a poly-Bergman type space of the Siegel domain $D_{2} \subset \mathbb{C}^{2}$.  
		Bounded measurable functions of  the form $c(\text{Im}\, \zeta_{1}, \text{Im}\, \zeta_{2} - |\zeta_1|^{2})$ are called nilpotent symbols.
		In this work we consider symbols of the form
		$a(\text{Im}\, \zeta_1) b(\text{Im}\, \zeta_2 -|\zeta_1|^{2})$, where both limits 
		$\lim\limits_{s\rightarrow 0^+} b(s)$ and  $\lim\limits_{s\rightarrow +\infty} b(s)$ exist, 
		and $a(s)$ belongs to the 
		set of piece-wise continuous functions on $\overline{\mathbb{R}}=[-\infty,+\infty]$ and 
		having one-side limit values at each point of a finite set $D\subset \mathbb{R}$.
		We prove that the $C^*$-algebra generated by all Toeplitz operators $T_{ab}$ is 
		isomorphic to $C(\overline{\Pi})$, where $\overline{\Pi}=\overline{\R} \times \overline{\R}_+$
		and $\overline{\mathbb{R}}_+=[0,+\infty]$.
	\end{abstract}

	%%%%%%%%%%%%%%%%%%%%%%%%%%%%%%%%%%%%%%%%%%%%%%%%%%%%
	%%%%%%%%%%%%%%%%%%%%%%%%%%%%%%%%%%%%%%%%%%%%%%%%%%%%
	%\input{seccion1}

	\section{Introduction.}
In the study of Toeplitz operators one of the common strategies consists in selecting a set of symbols 
$E \subset L^{\infty}$ in such a way that the algebra generated by Toeplitz operators with symbols in $ E $ can be described up to isomorphism, say, with an algebra of continuous functions or finding its spectrum.
In this paper, we study Toeplitz operators with nilpotent symbols and acting on 
a poly-Bergman type space of the Siegel domain $D_{2} \subset \mathbb{C}^{2}$. 
In \cite{QV-Ball1,QV-Ball2,LibroVasilevski}  the authors have fully described all commutative
$C^*$-algebras generated by Toeplitz operators with symbols invariant under the action
of a maximal Abelian subgroup of biholomorphisms and acting on the Bergman spaces of both the unit disk $\mathbb{D}$ and the Siegel domain  $D_n \subset \C^n$. 
For the unit disk, they discovered three families of symbols associated to 
commutative $C^*$-algebras of Toeplitz operators, while for the Siegel domain they found $n+2$ classes of symbols. 
Each class of symbols is invariant under the action of a maximal abelian group of biholomorphism.
Certainly, one can use these classes of symbols to  study Toeplitz operators acting on poly-Bergman type spaces
of the unit disk or the Siegel domain. 

Let $\Pi=\{z=x+iy \in \C \ : \  y>0\}$ be the upper half-plane. Toeplitz operators with vertical symbols, which depend  on $y=\im z$, and acting on  Bergman type  spaces  have been studied.
In \cite{Herrera,Herrera-Hutnik} the authors proved that the algebra generated by Toeplitz operators with vertical symbols and acting on the weighted  Bergman space  $\mathcal{A}_{\lambda}^{2}(\Pi)$ is isometrically isomorphic to the algebra  $VSO(\R_+)$ of all bounded functions that are very slowly oscillating on $\R_+$. 
Taking vertical symbols 
having limit values at $y=0$ and $y=\infty$, in \cite{R-SN} the authors found that  $\overline{\R}_+=[0,+\infty]$ is the spectrum of the algebra generated by all Toeplitz operators on the true-poly Bergman space $\mathcal{A}^{2}_{(n)}(\Pi)$. Similar research was made for Toeplitz operators on poly-Bergman spaces with homogeneous symbols (\cite{R-LL,R-RM-H}). Other works about it were made  in \cite{Hutnik,Hutnik-2,R-RM-V},  where the authors  studied 
Toeplitz operators acting on $\mathcal{A}^{2}_{(n)}(\Pi)$ from the point of view of wavelet spaces.
On the other hand, in \cite{Esmeral-Max-sq,Esmeral-Vas} the authors studied Toeplitz operators
on the Fock space $F_1^2(\C)$ with radial and bounded horizontal symbols, they found that spectral functions are uniformly continuous with respect to an adequate metric. Taking horizontal symbols having one-side limits at $x=\pm \infty$,	in \cite{A-C-R-N} the authors studied Toeplitz operators acting on poly-Fock spaces $F_{k}^{2}(\C)$, they found the spectrum of the $C^*$-algebra generated by such Toeplitz operators. 
Even though the authors described the $C^*$-algebras generated by  all spectral functions,  the spectrum of the  algebras is not fully understood in some cases, for this reason additional conditions on the symbols are imposed.

In \cite{QV-Ball1, QV-Ball2} the authors made remarkable research on the study of Toeplitz operators acting
on the Bergman space of the Siegel domain $D_{n} \subset \mathbb{C}^n$. In particular, they studied
  the $C^*$-algebra ${\cal T}_{{\cal N}_n}$   generated by all Toeplitz operators with
bounded nilpotent symbols, which are functions of the form  
$a(\zeta)=a(\im \zeta_{1},...,\im \zeta_{n-1}, \im \zeta_{n} - |\zeta'|^{2})$,
where $\zeta'=(\zeta_1,...,\zeta_{n-1})$.
Let us denote this kind of symbols by ${\cal N}_n$.
Although ${\cal T}_{{\cal N}_n}$ is commutative, it is too large, so it is impossible to figure out what its spectrum is. 
In particular,  in  \cite{Armando-V} the authors described the algebra generated  by Toeplitz operators acting on the weighted Bergman space $\mathcal{A}^2_{\lambda}(D_3)$ over three-dimensional Siegel domain $D_3$  using nilpotent symbols of the form 
$c(y_2)g(\im \zeta_3- |(\zeta_1, \zeta_2)|^2)$.  

The main purpose of the paper is to find the spectrum of the algebra generated by Toeplitz operators acting on the true-poly-Bergman type space $\mathcal{A}^{2}_{(L)}(D_{2})$ over  two-dimensional Siegel domain $D_2$ by selecting a particular set of nilpotent symbols.
In this sense, we just consider nilpotent symbols of the form 
$a(\im \zeta_1)$ and $b(\im \zeta_2 -|\zeta_1|^{2})$.
This paper is organized as follows.
In Section \ref{identificacion} we recall how poly-Bergman type spaces are defined for the Siegel domain, and how
they can be identified with a $L^2$-space through a Bargmann type transform.
In Section \ref{OT} we introduce Toeplitz operators acting on $\mathcal{A}^{2}_{(L)}(D_{2})$  with  nilpotent symbols,
we show that such Toeplitz operators are unitary equivalent to multiplication operators. 
In Section \ref{seccion-sb} we take symbols of the form
$b(\im \zeta_2 -|\zeta_1|^{2})$  for which both limits $\lim\limits_{s\rightarrow 0^+} b(s)$
and  $\lim\limits_{s\rightarrow +\infty} b(s)$ exist,
it is proved that  the $C^*$-algebra generated by all Toeplitz operators $T_b$ is 
isomorphic to $C(\overline{\R}_+)$, where  $\overline{\R}_+=[0,+\infty]$ is the one-point compactification of $[0,+\infty)$.
In Section \ref{Alg-a} we take nilpotent symbols of the form
$a(\im \zeta_1)$, where $a \in C(\overline{\R})$ and $\overline{\R}=[-\infty, +\infty]$ is the two-point compactification
of $\R$, we prove that the $C^*$-algebra generated by all Toeplitz
operators $T_a$ is isomorphic to $C(\triangle)$, where 
$\triangle= \overline{\Pi} /( \overline{\R} \times \{+\infty\})$ and $\overline{\Pi}=\overline{\R} \times \overline{\R}_+$. 
In Section \ref{symb-discont} we  describe the $C^*$-algebra generated by all Toeplitz operators 
$T_d$, where $d(\im \zeta_1)\in PC(\overline{\R},D)$  and $PC(\overline{\R},D)$ is the set of all 
piece-wise continuous functions on $\overline{\R}$  having one-side limit values at each point of 
a finite set $D\subset \R$. Finally, in Section \ref{T-ab} we describe the $C^*$-algebra generated by all Toeplitz operators 
$T_{ab}=T_a T_b=T_b T_a$.

%%%%%%%%%%%%%%%%%%%%%%%%%%%%%%%%%%%%%%%%%%%%%%%%%%%%
	%%%%%%%%%%%%%%%%%%%%%%%%%%%%%%%%%%%%%%%%%%%%%%%%%%%%
	
	\section{Poly-Bergman type spaces of the Siegel domain.}\label{identificacion}
	
In this section, we recall  some results obtained in  \cite{R-SN-D}, which are needed in our research about Toeplitz operators.
	Each $\zeta\in \mathbb{C}^n$ will be represented as an ordered pair
	$\zeta=(\zeta',\zeta_n)$, where $\zeta'=(\zeta_1,...,\zeta_{n-1}) \in \mathbb{C}^{n-1}$. Besides, the Euclidean norm function will
	be denoted by $|\cdot|$. The Siegel domain is defined by
	$$D_{n}=\{\zeta=(\zeta',\zeta_{n}) \in \C^{n-1} \times \C : \im \zeta_{n}- |\zeta'|^{2} >0\}.$$
	We will study Toeplitz operators acting on certain poly-Bergman type subspaces of 
  $L^{2}(D_{n}, d\mu_{\lambda})$,
	 where
	$d\mu_{\lambda}(\zeta)=(\im \zeta_{n}-|\zeta'|^{2})^{\lambda} d\mu (\zeta)$,  with $ \lambda >-1$,
	and $d\mu(\zeta)$ is the usual Lebesgue measure. 
	Once and for all, $L^2(X)$ means $L^2(X,dm)$, where $X$ is any subset of a Euclidean space and $dm$
	is the Lebesgue area measure on $X$.
	
	For each multi-index $L=(l_{1},...,l_{n}) \in \mathbb{N}^n$, %\cite{R-SN-D}
	the poly-Bergman type space $\mathcal{A}^{2}_{\lambda L}(D_{n})$ is the closed subspace of $L^{2}(D_{n}, d\mu_{\lambda})$ 
	consisting of all $L$-analytic functions, that is, all functions $f(\zeta)$ satisfying the equations
	\begin{eqnarray*}
		\left( \frac{\partial}{\partial \overline{\zeta}_{m}}-2i\zeta_{m}\frac{\partial}{\partial \overline{\zeta}_{n}}  \right)^{l_{m}} f 
		&=&0, \;\; 1\leq m \leq n-1 \\
		\left( \frac{\partial}{\partial \overline{\zeta}_{n}} \right)^{l_{n}}f 
		&=&0.
	\end{eqnarray*}
	
	In particular, for  $L=(1,...,1)$, $\mathcal{A}^{2}_{\lambda L}(D_{n})$ is just the Bergman space.
	Likewise, the anti-poly-Bergman type  space $\widetilde{\mathcal{A}}^2_{\lambda L}(D_{n})$ is defined to be the complex conjugate
	of $\mathcal{A}^{2}_{\lambda L}(D_{n}).$ 
	Thus, we introduce  true-poly-Bergman type spaces as follows:
	\begin{eqnarray*}
		\mathcal{A}^{2}_{\lambda(L)}(D_{n}) &=& \mathcal{A}^{2}_{\lambda L}(D_{n}) \ominus 
		\left( \sum_{m=1}^{n} \mathcal{A}^{2}_{\lambda,L-e_{m}}(D_{n}) \right),\\
		\widetilde{\mathcal{A}}^{2}_{\lambda(L)}(D_{n}) &=& \widetilde{\mathcal{A}}^{2}_{\lambda L}(D_{n}) \ominus 
		\left( \sum_{m=1}^{n} \widetilde{\mathcal{A}}^{2}_{\lambda,L-e_{m}}(D_{n}) \right),
	\end{eqnarray*}
	where  $e_{m}=(0,...,1,...,0)$ and the $1$ is placed at the $m$-entry. We assume that 
	$\mathcal{A}^{2}_{\lambda(L)}(D_{n})=\{0 \}$ whenever $L\in \Z^n \setminus \N^n$.
	
	In \cite{R-SN-D} the authors proved that $L^{2}(D_{n},d\mu_{\lambda})$ 
	equals to the direct sum of all the true-poly-Bergman type spaces:
	$$L^{2}(D_{n}, d\mu_{\lambda})= 
	\left( \bigoplus_{L\in \mathbb{N}^{n} } \mathcal{A}^{2}_{\lambda(L)}(D_{n}) \right) \bigoplus 
	\left( \bigoplus_{L\in \mathbb{N}^{n} } \widetilde{\mathcal{A}}^{2}_{\lambda(L)}(D_{n}) \right).$$
	The authors also proved that $\mathcal{A}^{2}_{\lambda(L)}(D_{n})$  is isomorphic and isometric to the tensor product
	$$L^{2}(\R^{n-1})\otimes {\cal H}_{l_{1}-1}\otimes \cdot\cdot\cdot\otimes {\cal H}_{l_{n-1}-1} \otimes L^{2}(\R_{+}) \otimes {\cal L}_{l_{n}-1},$$ 
	where $\mathbb{R}_+=(0,\infty)$.  Both ${\cal H}_{m}$ and ${\cal L}_{m}$ are one-dimensional spaces defined below. 
	Recall the Hermite and Laguerre polynomials:
	$$H_m(y):= (-1)^m e^{y^2} \frac{d^m}{dy^m} (e^{-y^2}), \quad 
	L_m^{\lambda}(y):= e^y \frac{y^{-\lambda}}{m!}  \frac{d^m}{dy^m} (e^{-y}y^{m+\lambda})$$
	for $m=0,1,2,...$ Recall also the Hermite and Laguerre functions 
		$$h_m (y)= \frac{(-1)^m}{(2^n\sqrt{\pi}n!)^{1/2}} H_m(y) e^{-y^2/2}, \quad 
		\ell_m^{\lambda} (y)=(-1)^m c_m L_m^{\lambda}(y) e^{-y/2},$$
		where $c_m=\sqrt{m!/\Gamma(m+\lambda+1)}$ and $\Gamma$ is the usual Gamma function.
	It is well known that $\{h_m \}_{m=0}^{\infty}$ and $\{\ell_m^{\lambda} \}_{m=0}^{\infty}$ are orthonormal
	bases for $L^2(\R)$ and $L^2(\mathbb{R}_+, y^{\lambda} dy)$, respectively. 
	Finally, ${\cal H}_m=\text{span}\{h_m\}$ and ${\cal L}_m=\text{span}\{\ell_m^{\lambda}\}$.
			
 In this work we restrict ourselves to the study of Toeplitz operators acting 
	on the true-poly-Bergman type spaces over two-dimensional Siegel domain $D_{2}$ 
	with the Lebesgue measure $d\mu$ ($\lambda=0$). Henceforth,
	the space $\mathcal{A}^{2}_{0(L)}(D_{2})$  will be simply denoted by $\mathcal{A}^{2}_{(L)}(D_{2})$;
	similarly, $\ell_m(y)$ and $L_m(y)$ stand for $\ell_m^0(y)$ and $L_m^0(y)$, respectively.
	The true-poly-Bergman type space $\mathcal{A}^{2}_{(L)}(D_{2})$ can be identified with $L^{2}(\R \times \R_{+})$ through 
	a Bargmann type transform (\cite{R-SN-D}), such identification fits to the study of Toeplitz operators with nilpotent symbols. 
	Several operators are needed to define such identification.  To begin with,
	we introduce the flat domain $\Do=\C \times \Pi $, where $\Pi= \R \times \R_{+} \subset \C$. 
	Then $\Do$ can be identified with  $D_{2}$ using  the mapping
	$$\kappa: \Do \ni w=(w_{1},w_{2}) \longmapsto \zeta=(w_{1}, w_{2}+i|w_{1}|^{2}) \in D_2.$$

Thus we have the unitary operator
	$U_{0}: L^{2}(D_{2}, d\mu) \longrightarrow L^{2}(\Do, d\eta)$
	given by
	$$ (U_{0}f)(w)=f(\kappa(w)),$$
	where $d\eta(w)= d\mu(w)$.  Take 
	$w=(w_{1}, w_{2}) \in \C \times \Pi$, with $w_{m}=u_m +i v_m$ and  $m=1,2$. We identify 
	$w=(u_1+ iv_1, u_2+iv_2)$ with $(u_1,v_1, u_2,v_2)$. Then
	$$L^{2}(\Do, d\eta)= L^{2}(\R, du_1)\otimes L^{2}(\R,  dv_{1}) \otimes L^{2}(\R, du_2)\otimes L^{2}(\R_{+}, dv_2).$$
	
Introduce
	$$U_{1}= F\otimes I \otimes F \otimes I,$$
	%$$ U_{2}= F \otimes I \otimes I \otimes I,$$
	where $F$ is the Fourier transform acting on $L^{2}(\R)$ by the rule
	$$(Fg)(t)=\frac{1}{\sqrt{2 \pi}} \int_{-\infty}^{\infty} g(x) e^{-itx} dx.$$
	
	Consider now the following two mappings acting on $\Do$:
	$$\psi_1: \xi=(\xi_{1}, t_{2} +i   s_2) \longmapsto w=(\xi_1, t_{2}+ i\frac{s_{2}}{2|t_{2}|}),$$
	$$\psi_2: z=(x_1+iy_1, z_2) \longmapsto \xi= \left(\sqrt{|x_{2}|}(x_{1}+y_{1})+ i \frac{1}{2\sqrt{|x_{2}|}}(-x_{1}+y_{1}), z_{2}\right),$$
	where $\xi=(\xi_1,\xi_2)$, $z=(z_1,z_2) \in \Do$,  $\xi_m=t_m+is_m$  and $z_m=x_m+iy_m$.
	Both functions $\psi_1$ and $\psi_2$ lead to the following unitary operators
	acting on $L^{2}(\Do, d\eta)$:
	$$(V_1f)(\xi) = \frac{1}{(2|t_{2}|)^{ 1/2}} f(\psi_1(\xi)), \quad (V_{2}g)(z)= g(\psi_2(z)) .$$

	Henceforth  $L=(l_{1}, l_{2})=(j,k)$.
	
	\begin{thm}[\cite{R-SN-D}]\label{puro-iso}
		The operator $U=V_{2}V_{1}U_{1}U_{0}$ %: L^{2}(D_{2}, d\mu) \longrightarrow L^{2}(\Do, d\eta)$ 
		is unitary, and maps
		$L^{2}(D_{2}, d\mu)$ onto the space
		$$L^{2}(\Do, d\eta)= L^{2}(\R,dx_1)\otimes L^{2}(\R,dy_1) \otimes L^{2}(\R,dx_2)\otimes L^{2}(\R_{+},dy_2).$$
		For each $L=(j,k) \in \N^{2}$, the operator $U$ restricted to $\mathcal{A}^{2}_{(L)}(D_{2})$ is an isometric isomorphism onto
		the space
		$$\mathcal{H}_{(L)}^{+}= L^{2}(\R) \otimes \text{span} \{h_{j-1}(y_1)\}  \otimes L^{2}(\R_{+}) \otimes \text{span} \{\ell_{k-1}(y_2)\}. $$ 
	\end{thm}

	\vspace{.3cm}
	Introduce the isometric linear embedding
	$R_{0(L)}: L^{2}(\R \times \R_{+}) \longrightarrow L^{2}(\Do)$
	defined  by
	\begin{equation*}
		(R_{0(L)}g)(x_1,y_1,x_2,y_2)= \chi_{\R_{+}}(x_2)\, g(x_{1}, x_{2}) \, h_{j-1}(y_{1}) \ell_{k-1}(y_{2}).
	\end{equation*}
	Of course $\mathcal{H}_{(L)}^{+}$ is the range of $R_{0(L)}$, and it is also
	the image of $\mathcal{A}^{2}_{(L)}(D_{2})$ under $U$. Thus, the operator
	$$R_{(L)}=R_{0(L)}^{*}U: L^{2}(D_{2}) \longrightarrow L^{2}(\R \times \R_{+}),$$
	isometrically maps the true-poly-Bergman type space $\mathcal{A}^{2}_{(L)}(D_{2})$ onto  $L^{2}(\R \times \R_{+})$.
	Therefore, $R_{(L)}R_{(L)}^{*}=I$ and $R_{(L)}^{*}R_{(L)}=B_{(L)}$, where
	$B_{(L)}$ is the orthogonal projection from $L^{2}(D_{2})$ onto $\mathcal{A}^{2}_{(L)}(D_{2})$.
	In addition, the operator $R_{(L)}^{*}=U^*R_{0(L)}$ plays the role of the Segal-Bargmann transform 
	for the true-poly-Bergman type space $\mathcal{A}^{2}_{(L)}(D_{2})$,
	where the adjoint operator
	$R_{0(L)}^{*}: L^{2}(\Do) \longrightarrow L^{2}(\R \times \R_{+})$
	is given by
	\begin{equation*}
	(R_{0(L)}^*f)(x_{1}, x_{2})
	= %\chi_{\R_{+}}(x_{2}) 
	\int_{\R}\int_{\R_{+}} h_{j-1}(y_{1}) \ell_{k-1}(y_{2}) f(x_{1},y_{1}, x_{2}, y_{2})  dy_{2} dy_{1},
	\end{equation*}
	with $(x_1,x_2)\in \mathbb{R} \times \mathbb{R}_+$.

	%%%%%%%%%%%%%%%%%%%%%%%%%%%%%%%%%%%%%%%%%%%%%%%%%%%%
	%%%%%%%%%%%%%%%%%%%%%%%%%%%%%%%%%%%%%%%%%%%%%%%%%%%%
	%\input{seccion3}
	
	\section{Toeplitz operators with nilpotent symbols.} \label{OT}
	
	In this section, we study Toeplitz operators with nilpotent symbols and acting on the true-poly-Bergman type space
	$\mathcal{A}^{2}_{(L)}(D_{2})$. In \cite{LibroVasilevski} the author  has widely developed the theory of Toeplitz operators
	on the Bergman spaces, and the author's techniques can be applied to the study of Toeplitz operators
	acting on $\mathcal{A}^{2}_{(L)}(D_{2})$. To begin with, a function 
	$c \in L^{\infty}(D_{2}, d\mu)$ is said to be  a nilpotent symbol if it has the form 
	$c(\zeta_{1}, \zeta_{2})=c(\im \zeta_{1},\im \zeta_{2}- |\zeta_{1}|^{2})$. 
	Then the Toeplitz operator acting on $\mathcal{A}^{2}_{(L)}(D_{2})$,
	with  nilpotent symbol $c(\zeta)$, is defined by
	$$(T_cf)(\zeta)=(B_{(L)}(cf))(\zeta),$$
	where $B_{(L)}$ is the orthogonal projection from $L^{2}(D_{2})$ onto $\mathcal{A}^{2}_{(L)}(D_{2})$.
	The  Bargmann-type operator  $R_{(L)}$ identifies the space  $\mathcal{A}^{2}_{(L)}(D_{2})$ with
	$L^{2}(\R\times \R_{+})$, and it fits properly in the study of the Toeplitz operator $T_c$.

	\begin{thm}\label{equiv-unitaria}
		Let $c$ be a nilpotent symbol.
		Then the Toeplitz operator $T_{c}$ is unitary equivalent to the multiplication operator $\gamma^{c}I=R_{(L)}T_cR_{(L)}^*$,  
		where $\gamma^{c} :\R \times \R_{+} \rightarrow \C$ is given by
		\begin{equation}\label{gamma-c}
	\gamma^{c}(x_1,x_2)=
\int_{\R}\int_{\R_{+}} c\left(\frac{-x_{1}+y_{1}}{2\sqrt{x_{2}}} ,\frac{y_{2}}{2x_{2}}\right) (h_{j-1}(y_{1}))^{2}
(\ell_{k-1}(y_{2}))^{2}dy_{2}dy_{1}.
		\end{equation}
	\end{thm}

	\begin{proof}
		We have
		\begin{eqnarray*}
			R_{(L)}T_cR_{(L)}^*
			&=& R_{(L)}B_{(L)} (cI) R_{(L)}^*\\
			&=& R_{(L)}R_{(L)}^{*}R_{(L)} (cI) R_{(L)}^*\\
			&=&R_{(L)} (cI) R_{(L)}^*\\
			&=&R_{0(L)}^{*}V_{2}V_{1}U_{1}U_{0} (cI) U_0^{-1}U_1^{-1}V_{1}^{-1}V_{2}^{-1} R_{0(L)}.
			% &=&R_{0(L)}^{*}V_{2}V_{1}U_{1}a(k(w))U_1^{-1}V_{1}^{-1}V_{2}^{-1} R_{0(L)}
		\end{eqnarray*}

	Recall that  $\zeta=\kappa(w)=(w_{1}, w_2 + i |w_{1}|^{2})$, where $w=(w_{1}, w_2) \in \Do$ and $w_m=u_m +iv_m$.
		For $g \in L^{2}(\Do)$,
		$$(U_{0} (cI) U_{0}^{-1}g)(w)= c(\kappa(w))(U_{0}^{-1}g)(\kappa(w))= c(\kappa(w))g(w).$$
		That is, $U_{0} (cI) U_{0}^{-1}=c(\kappa(w)) I$, where 
		$c(\kappa(w))=c(v_1,v_2)$.
		It is easy to see that
		$U_{1} (c(v_1,v_2)I) U_{1}^{-1}= c(v_1,v_2)I,$
		$$V_{1}(cI)V_{1}^{-1}= c(s_{1}, \frac{s_{2}}{2|t_{2}|})I$$ %\quad 
		and
		$$V_{2}V_1(cI) V_1^{-1} V_{2}^{-1} = c\left( \frac{-x_{1}+y_{1}}{2\sqrt{|x_{2}|}} ,\frac{y_{2}}{2|x_{2}|}\right) I.$$
		Thus
		\begin{eqnarray*}
			R_{(L)}T_cR_{(L)}^*&=& R_{0(L)}^{*}c\left( \frac{-x_{1}+y_{1}}{2\sqrt{|x_{2}|}}  ,\frac{y_{2}}{2|x_{2}|}\right) I R_{0(L)}\\
			&=&\gamma^{c}(x_{1}, x_{2}) I, 
		\end{eqnarray*}
		where $\gamma^{c}(x_1,x_2)$ is given in (\ref{gamma-c}).
	\end{proof}

 By Theorem \ref{equiv-unitaria},  the $C^*$-algebra generated by all Toeplitz operators $T_c$ is commutative (See \cite{QV-Ball1, LibroVasilevski}), but  its spectrum is difficult to figure out what it is.
For this reason, we assume certain continuity conditions on the nilpotent symbols in order to describe the spectrum of the subalgebra generated 
by the Toeplitz operators.  We will split our research into two cases concerning the symbols.
Firstly, we study Toeplitz operators with symbols of the form
	$b(\im \zeta_{2}- |\zeta_{1}|^{2})$, for which
	\begin{equation}\label{sim-b}
	\gamma^{b}(x_{1}, x_{2})= \gamma^{b}( x_{2})= \int_{\R_{+}} b\left(\frac{y_{2}}{2x_{2}}\right) (\ell_{k-1} (y_{2}))^{2}  dy_{2}.
	\end{equation}
	Secondly, we analyze Toeplitz operators with symbols of the form	$a(\im \zeta_{1})$,
	for which
	\begin{equation}\label{sim-a}
		\gamma^{a}(x_{1}, x_{2}) = \int_{\R} a\left( \frac{-x_{1}+y_{1}}{2\sqrt{x_{2}}} \right)  (h_{j-1}(y_{1}))^{2} dy_{1}.\\
	\end{equation}
As mentioned above,
the $C^*$-algebra generated by all Toeplitz operators $T_a$ is still complicated to be fully described despite its commutative property.
	 Fortunately, the $C^*$-algebra generated by all operators $T_a$ can be described when the symbols
	  $a$ are taken to be continuous on $\overline{\mathbb{R}}=[-\infty,+\infty]$, where  $\overline{\R}$ is the two-point compactification of $\R$.   
	  Even more, the $C^*$-algebra of Toeplitz operators $T_a$ can be still described for symbols 
	  having finitely many jump discontinuities, as shown in Section \ref{symb-discont}. Finally, we analyze Toeplitz operators with symbols of the form $c(\zeta_{1}, \zeta_{2})=a(\im \zeta_{1})b(\im \zeta_{2}- |\zeta_{1}|^{2})$.

	%%%%%%%%%%%%%%%%%%%%%%%%%%%%%%%%%%%%%%%%%%%%%%%%%%%%
	%%%%%%%%%%%%%%%%%%%%%%%%%%%%%%%%%%%%%%%%%%%%%%%%%%%%	
	
	\section{Toeplitz operators with symbols $b(\im \zeta_{2}- |\zeta_{1}|^{2})$.}\label{seccion-sb}
	
	In this section, we study the $C^*$-algebra generated by all Toeplitz operators $T_b$ with symbols of the form $b(\im \zeta_{2}- |\zeta_{1}|^{2})$, where $b(y)$ has limit values at $y=0,+\infty$.
Under this continuity condition,  we will see that $\gamma^{b}$  is continuous on 
	$\overline{\Pi}:=\overline{\R} \times \overline{\R}_+$, where  
	$\overline{\R}_+= [0, +\infty]$ is the two-point compactification of $\R_+=(0,+\infty)$.
	Apply the change of variable $y_2 \mapsto 2x_2y_2$ in the integral representation of $\gamma^{b}$, 
	then
	$$\gamma^{b}(x_{1}, x_{2})=\gamma^{b}(x_{2})= 2x_{2} \int_{\R_{+}} b(y_{2})   (\ell_{k-1} (2x_{2}y_{2}))^{2}  dy_{2}.$$
	Actually $\gamma^b$ depends only  on the variable $x_2$, and is continuous on $\R_+$ because of the
	continuity of $\ell_{k-1}(y)$  and the Lebesgue dominated convergence theorem.

	Let $L^{\infty}_{\{0, +\infty\}}(\R_{+})$ denote the subspace of $L^{\infty}(\R_{+})$ consisting of all functions having limit values at
 $0$ and $+\infty$.
	For $b \in L^{\infty}_{\{0, +\infty\}}(\R_{+})$, define
	\begin{equation*}
		b_{0}:=\lim_{y\rightarrow 0^{+}} b(y), \quad b_{\infty}:=\lim_{y\rightarrow +\infty} b(y).
	\end{equation*}

It is worth mentioning that $\gamma^b$ was obtained in \cite{R-SN} as the spectral function of a Toeplitz operator
acting on a true-poly-Bergman space of the upper half-plane.  Thus, we have at least two scenarios in which $\gamma^b$ appears
as a spectral function. 

\begin{lem}[\cite{R-SN}]\label{lim-sym-b}
	Let $b \in L^{\infty}_{\{0, +\infty\}}(\R_{+})$. 
	Then the spectral function $\gamma^{b}$ satisfies
	$$b_{\infty}=\lim_{x_{2}\rightarrow 0^{+}} \gamma^{b}(x_{2}), \quad b_{0}=\lim_{x_{2}\rightarrow + \infty} \gamma^{b}(x_{2}).$$
\end{lem}

According to Lemma \ref{lim-sym-b} and Theorem 4.8 in \cite{R-SN} we have the following 

\begin{thm}
For  $b\in L^{\infty}_{\{0, +\infty\}}(\R_{+})$, the spectral function $\gamma^{b}(x_{2}) $ is continuous on $\overline{\R}=[0, +\infty]$.	
	The $C^{*}$-algebra generated by all functions $\gamma^{b}$, with $b\in L^{\infty}_{\{0, +\infty\}}(\R_{+})$,  is isomorphic 
	and isometric to the algebra  $C[0, \infty]$. That is, the  $C^*$-algebra generated by all Toeplitz operators $T_b$,
	with $b(\im \zeta_{2}- |\zeta_{1}|^{2}) \in L^{\infty}_{\{0, +\infty\}}(\R_{+})$,
	 is isomorphic to $C[0, \infty]$, where the  isomorphism is defined on the generators by 
	$$T_b  \longmapsto \gamma^b.$$
\end{thm}

\vspace{.3cm}
Obviously the spectral function $\gamma^b(x_2)$ is defined and continuous on $\overline{\Pi}$ but 
it is constant along each horizontal straight line. 
Thus,  $\gamma^b$ is identified with a continuous function on the quotient space
$\overline{\Pi}/\overline{\R}$, which is homeomorphic to $\overline{\R}_+$.

%%%%%%%%%%%%%%%%%%%%%%%%%%%%%%%%%%%%%%%%%%%%%%%%%%%%
%%%%%%%%%%%%%%%%%%%%%%%%%%%%%%%%%%%%%%%%%%%%%%%%%%%%

	\section{Toeplitz operators with continuous symbols $a(\im \zeta_{1})$.}\label{Alg-a}

In this section, we study the $C^*$-algebra generated by all Toeplitz operators  $T_a$,
where symbols $a(\im \zeta_{1})$ are taken to be continuous on $\overline{\mathbb{R}}$. 
Once again, such a $C^*$-algebra can be identified with the algebra of all continuous functions on a quotient space
of $\overline{\Pi}$.  Henceforth, $(x_1, x_2)$ will denote points in $\overline{\Pi}$ instead of intervals.

It is fairly simple to see that $\gamma^{a}$ is continuous on $\Pi$.
Take the change of variable $y_1 \mapsto 2\sqrt{x_2} y_1+x_1$ in the integral representation of  $\gamma^{a}$, then
$$\gamma^{a}(x_{1}, x_{2})= 2\sqrt{x_{2}} \int_{\R} a(y_{1})(h_{j-1}(2\sqrt{x_{2}}y_{1}+x_{1} ))^{2} dy_{1}.$$
The function $\gamma^a$ is continuous at each point $(x_{1},x_{2}) \in \Pi$ because of the continuity of
$h_{j-1}$ and the Lebesgue dominated convergence theorem.
Next, we will prove that  $\gamma^a$ has one-side  limit value at each point of $\R \times \{0\}$.
For  $a \in L^{\infty}(\R)$ we introduce the notation
\begin{equation}\label{limit-a-infty}
	a_{-}= \lim_{y\rightarrow -\infty} a(y) \quad \text{and} \quad 	a_{+}=\lim_{y\rightarrow +\infty} a(y)
\end{equation}
if such limits exist.

\begin{lem}\label{lim-gamma-x0-0}
	Let $a\in L^{\infty}(\R)$, and suppose that $a(y)$ converges at  $y=\pm \infty$.
	Then for each $x_{0} \in \R$, the spectral function $\gamma^{a}$  satisfies
	\begin{equation}\label{limit-R}
	\lim_{(x_{1}, x_{2}) \rightarrow (x_{0}, 0)} \gamma^{a}(x_{1}, x_{2})= a_{-}\int_{-\infty}^{x_{0}} (h_{j-1}(y_{1}))^{2} dy_{1}
	+  a_{+}\int_{x_{0}}^{\infty} (h_{j-1}(y_{1}))^{2} dy_{1}.
	\end{equation}
\end{lem}

\begin{proof}
Let $A$ denote the  right-hand side of equality (\ref{limit-R}).
	Take $\epsilon > 0$. We will prove that there exist $\delta >0$ such that
	$|\gamma^{a}(x_{1}, x_{2})-A| < \epsilon$ whenever $|x_{1}-x_{0}|< \delta$ and $0<x_{2}< \delta$.
	Note that $|a_-|,|a_+| \leq \|a\|_{\infty}$.
	Since $\int_{-\infty}^{\infty}(h_{j-1}(y_{1}))^{2} dy_{1}=1$,
	there exists  $\delta_{1} >0$  such that
	$$\|a\|_{\infty} \int_{-\delta_{1}+x_{0}}^{\delta_{1}+x_{0}}(h_{j-1}(y_{1}))^{2} dy_{1} < \frac{\epsilon}{5 }.$$
	Then
	\begin{eqnarray*} 
		I
		&:=&|\gamma^{a}(x_{1}, x_{2})-A| \\
		& = & \bigg| \int_{-\infty}^{\infty} a\left(\frac{-x_{1}+y_{1}}{2\sqrt{x_{2}}}  \right)  (h_{j-1}(y_{1}))^{2} dy_{1}\\ 
		& - & a_{-}\int_{-\infty}^{x_{0}} (h_{j-1}(y_{1}))^{2} dy_{1}- a_{+}\int_{x_{0}}^{\infty} (h_{j-1}(y_{1}))^{2} dy_{1}\bigg| \\
		& \leq & \int_{-\infty}^{-\delta_{1}+x_{0}}\left| a\left( \frac{-x_{1}+y_{1}}{2\sqrt{x_{2}}} \right) -a_{-}\right|  (h_{j-1}(y_{1}))^{2} d{y_{1}} \\
		&+& |a_{-}| \int_{-\delta_{1}+x_{0}}^{x_{0}}(h_{j-1}(y_{1}))^{2} dy_{1} + |a_{+}| \int_{x_{0}}^{\delta_{1}+x_{0}}(h_{j-1}(y_{1}))^{2} dy_{1}\\
		&+& \int_{-\delta_{1}+x_{0}}^{\delta_{1}+x_{0}} \left| a\left( \frac{-x_{1}+y_{1}}{2\sqrt{x_{2}}} \right)  (h_{j-1}(y_{1}))^{2} \right| d{y_{1}} \\
		&+&\int_{\delta_{1}+x_{0}}^{\infty}\left| a\left( \frac{-x_{1}+y_{1}}{2\sqrt{x_{2}}} \right) -a_{+}\right|  (h_{j-1}(y_{1}))^{2}d{y_{1}} \\
		& \leq &     \max_{-\infty< y_{1}< -\delta_{1}+x_{0}}\left| a\left( \frac{-x_{1}+y_{1}}{2\sqrt{x_{2}}} \right) -a_{-}  \right| +
		\frac{ 3\epsilon}{5} \\
		&+&\max_{\delta_{1}+ x_{0}< y_{1}< \infty}\left| a\left( \frac{-x_{1}+y_{1}}{2\sqrt{x_{2}}} \right) -a_{+}  \right|. 
	\end{eqnarray*}
	
We have assumed that $a(y)$ converges  at $\pm \infty$, then there exists $N>0$ such that 
	$|a(y)-a_{-}|< \epsilon/5$ and $|a(y)-a_{+}|< \epsilon/5$ for $|y|> N$. 
	Let $\delta=\min \{\delta_{1}/2, \delta_1^2/(16N^2) \}$.
	Then we have $\frac{1}{2\sqrt{x_{2}}}|-x_{1}+y_{1}| > N$ if $|x_1-x_0| < \delta$,
	$ 0 <x_2 < \delta$ and $|y_1-x_0|\geq \delta_{1}$. 
	Thus,
	$$\max_{-\infty< y_{1}< -\delta_{1}+x_{0}}\left| a\left(\frac{-x_{1}+y_{1}}{2\sqrt{x_{2}}} \right) -a_{-}  \right|< \frac{\epsilon}{5} $$
	and
	$$\max_{\delta_{1}+ x_{0}< y_{1}< \infty}\left| a\left(\frac{-x_{1}+y_{1}}{2\sqrt{x_{2}}} \right)  -a_{+}  \right| < \frac{\epsilon}{5}.$$
Finally, we conclude that $|\gamma^{a}(x_{1}, x_{2})-A|<\epsilon$ whenever 
		$|x_{1}-x_{0}|< \delta$ and $0<x_{2}< \delta$.
\end{proof}

In general, $\gamma^a(x_1,x_2)$ does not converge at each point  $x=(\pm \infty, +\infty)\in \overline{\Pi}$;  
however, $\gamma^{a}(x_{1}, x_{2})$  has limit values along the parabolas $x_{2}=\alpha (x_{1}^{2}+1)$, with $\alpha >0$.
We will define a bijective mapping  $\Phi: \Pi \rightarrow \Pi$ so that
  $\phi^{a}=\gamma^a \circ \Phi^{-1}$ will be  a continuous mapping on
 $\overline{\Pi}=\overline{\R}\times\overline{\R}_+$ with the usual topology.

	%%%%%%%%%%%%%%%%%%%%%%%%%%%%%%%%%%%%%%%%%%%%%%%%%%%%
%\input{seccion4}

\subsection{Modified spectral function for $T_a$.}

Let $\Phi: \Pi \longrightarrow \Pi$ be the mapping
$$%(t_{1}, t_{2})=
\Phi(x_{1}, x_{2})=\left(x_{1}, \frac{x_{2}}{x_{1}^{2} + 1}\right).$$
It is easy to see that $\Phi^{-1}(t_1,t_2)=(t_1, (t_1^2+1)t_2)$. 
Concerning the spectral properties of $T_a$, the function  $\phi^{a}:=\gamma^{a} \circ \Phi^{-1}$ is as important as $\gamma^a$, 
but $\phi^a$ behaves much better than $\gamma^a$, at least for $a$ continuous on $\overline{\R}$. From now on we take $\phi^a$ 
as the spectral function for the Toeplitz operator $T_a$. A direct computation shows that
$$\phi^{a}(t_{1}, t_{2})=%:=(\gamma^{a} \circ \Phi^{-1})(t_{1}, t_{2})=
\int_{-\infty}^{\infty} a\left(\frac{-t_{1}+s_{1}}{2\sqrt{t_{2}(t_{1}^{2} +1)}}\right)  (h_{j-1}(s_{1}))^{2} ds_{1}.$$

Both $\Phi$ and $\Phi^{-1}$ are continuous on $\Pi$. 
Besides, the spectral function $\phi^{a}=\gamma^{a} \circ \Phi^{-1}$ is continuous on $\Pi$ because $\gamma^{a}$ is. Since 
$\Phi^{-1}(t_1, 0)=(t_1,0)$,  we have $\phi^{a}(t_{1}, 0)=\gamma^{a}(x_1, 0)$. By Lemma \ref{lim-gamma-x0-0},
$\phi^{a}$ is also continuous on $\mathbb{R} \times \{0 \}$.

\begin{thm}
	For $a(\im \zeta_1) \in C(\overline{\R})$, the spectral function $\phi^a: \Pi \rightarrow \mathbb{C}$ 
	can be extended continuously to $\overline{\Pi}=\overline{\R} \times \overline{\R}_+$.
\end{thm}

\begin{proof}
Follows from Lemmas \ref{lim-gamma-x0-0},  \ref{lim-phi-infty-0}, \ref{lim-phi-infty-t0}  and  \ref{lim-phi-t0-infty} below.
\end{proof}

For any domain $X\subset \R^m$  and a function $\varphi: X\rightarrow \mathbb{C}$,
we write $\varphi(x_0)$ to mean the limit value of $\varphi$ at $x_0$, even if $x_0$ does not belong to $X$.
For example, $a(+\infty)$ means $\lim\limits_{x\rightarrow +\infty} a(x)$.

\begin{lem}\label{lim-phi-infty-0}
	Let $a\in L^{\infty}(\R)$, and suppose that $a(y)$ converges at the points  $y=\pm \infty$. Then $\phi^{a}$ satisfies
	$$\lim_{(t_{1}, t_{2}) \rightarrow (+\infty, 0)} \phi^{a}(t_{1}, t_{2})= a(-\infty).$$
	That is, for $\epsilon >0$ there exists $\delta>0$ and $N>0$ such that $|\phi^{a}(t_{1}, t_{2})|<\epsilon$
	whenever $0< t_{2}< \delta$ and $t_{1}>N$. Analogously,
	$$ \lim_{(t_{1}, t_{2}) \rightarrow (-\infty, 0)} \phi^{a}(t_{1}, t_{2})=a(+\infty).$$	
\end{lem}

\begin{proof}
	Suppose that $a(-\infty)=0$. Let  $\epsilon>0$.
	Since $h_{j-1} \in L^{2}(\R)$, there exists $s_{0}>0$ such that
	$$\|a\|_{\infty} \int_{s_{0}}^{\infty}(h_{j-1}(s_{1}))^{2} ds_{1} < \frac{\epsilon}{2}.$$
	Take into account $\int_{-\infty}^{\infty}(h_{j-1}(s_{1}))^{2} ds_{1}=1$ in the following computation
	\begin{eqnarray*}
		|\phi^{a}(t_{1}, t_{2})| 
		& = & \left| \int_{-\infty}^{\infty} a\left(\frac{-t_{1}+s_{1}}{2\sqrt{t_{2}(t_{1}^{2} +1)}}\right) (h_{j-1}(s_{1}))^{2} ds_{1}  \right|\\
		& \leq & \int_{-\infty}^{s_{0}} \left| a\left(\frac{-t_{1}+s_{1}}{2\sqrt{t_{2}(t_{1}^{2} +1)}}\right) (h_{j-1}(s_{1}))^{2} \right| d{s_{1}}\\
		&+& \int_{s_{0}}^{\infty}\left| a\left(\frac{-t_{1}+s_{1}}{2\sqrt{t_{2}(t_{1}^{2} +1)}}\right)  (h_{j-1}(s_{1}))^{2}\right| d{s_{1}}\\
		& \leq &     \max_{-\infty< s_{1}< s_{0}}\left| a\left( \frac{-t_{1}+s_{1}}{2\sqrt{t_{2}(t_{1}^{2} +1)}} \right)  \right| +
		\frac{ \epsilon}{2}.
	\end{eqnarray*}
	
	Since $a(s)$ converges to zero at $-\infty$, there exists $N_{1}>0$ such that $|a(s)|< \epsilon/2$ for $-s> N_{1}$. 
	Take $\delta= 1/(16N_{1}^{2})$. Then we have  $\frac{1}{2\sqrt{t_2}} > 2N_1$ for $0<t_2<\delta$.
	On the other hand, 	assume  $t_{1}> s_0$ 	and $-\infty < s_1 < s_{0}$.  Then
	$$\frac{t_{1}- s_{1}}{\sqrt{t_{1}^{2} +1}} > \frac{t_{1}- s_{0}}{\sqrt{t_{1}^{2} +1}}.$$
	The right-hand side of this inequality converges to $1$ when 
	$t_1$ tends to $+\infty$, thus there exists
	$N_2>s_0$ such that $(t_{1}- s_{0})/\sqrt{t_{1}^{2} +1}>1/2$ for $t_1>N_2$. Consequently,
	$$N_1=2N_1\frac{1}{2} < \frac{1}{2\sqrt{t_2}} \frac{t_{1}- s_{0}}{\sqrt{(t_{1}^{2} +1)}} <
	\frac{t_{1}-s_{1}}{2\sqrt{t_{2}(t_{1}^{2} +1)}}. $$
	For $0<t_2<\delta$ and $t_1>N:=\max\{ s_0,N_2 \}$ we have
	$$\left| a\left(\frac{-t_{1}+s_{1}}{2\sqrt{t_{2}(t_{1}^{2} +1)}}\right)  \right| < \frac{\epsilon}{2}.$$
	
	We define $\tilde{a}(s)=a(s)-a_2$ in the case $a(-\infty) \neq 0$, where $a_2:=a(-\infty)$ is a constant. 
 Note that $\tilde{a}(s)$  converges to zero at $-\infty$, 
	and $\phi^{a_1+a_2}=\phi^{a_1}+\phi^{a_2}$ for any nilpotent symbols $a_1$ and $a_2$.
	Then
	\begin{eqnarray*}
		\lim_{(t_{1}, t_{2}) \rightarrow (+\infty, 0)} \phi^{a}(t_{1}, t_{2})
		& = & \lim_{(t_{1}, t_{2}) \rightarrow (+\infty, 0)} \phi^{\tilde{a}+a_2}(t_{1}, t_{2}) \\
		& = &  \lim_{(t_{1}, t_{2}) \rightarrow (+\infty, 0)} \phi^{\tilde{a}}(t_{1}, t_{2})
		+
		a_2 \int_{-\infty}^{\infty} (h_{j-1}(s_{1}))^{2} ds_{1}\\
		& = & a(-\infty).
	\end{eqnarray*}

Finally, the limit of $\phi^a(t_1,t_2)$ at $(-\infty,0)$ can be proved analogously. 
\end{proof}

\begin{lem}\label{lim-phi-infty-t0}
	Let $t_{0} \in \R_+$. If $a\in L^{\infty}(\R)$ is continuous at $-1/(2\sqrt{t_0})$, 
	then  the spectral function $\phi^{a}$ satisfies
	$$\lim_{(t_{1}, t_{2}) \rightarrow (+\infty, t_{0})} \phi^{a}(t_{1}, t_{2})= a\left(-\frac{1}{2\sqrt{t_{0}}}\right).$$
	Analogously, if $a$ is continuous at $1/(2\sqrt{t_0})$,  then
	$$ \lim_{(t_{1}, t_{2}) \rightarrow (-\infty, t_{0})} \phi^{a}(t_{1}, t_{2})=a\left(\frac{1}{2\sqrt{t_{0}}}\right).$$	
\end{lem}

\begin{proof}
	Suppose that $a$ converges to zero at $-1/(2\sqrt{t_{0}})$. 
	Let $\epsilon >0$.	Since $h_{j-1} \in L^{2}(\R)$, there exists $s_{0}>0$  such that
	\begin{equation} \label{acotar-colas}
		\|a\|_{\infty}\int_{-\infty}^{-s_{0}}(h_{j-1}(s_{1}))^{2} ds_{1} < \frac{\epsilon}{3}, \quad
		\|a\|_{\infty}\int_{s_{0}}^{\infty}(h_{j-1}(s_{1}))^{2} ds_{1} < \frac{\epsilon}{3}.
	\end{equation}
	Take into account $\int_{-\infty}^{\infty}(h_{j-1}(s_{1}))^{2} ds_{1}=1$ in the following computation
	\begin{eqnarray*}
		|\phi^{a}(t_{1}, t_{2})|
		& \leq & \int_{-\infty}^{-s_{0}} 
		\left| a\left(\frac{-t_{1}+s_{1}}{2\sqrt{t_{2}(t_{1}^{2} +1)}}\right) \right| (h_{j-1}(s_{1}))^{2}  d{s_{1}} \\
		&+& \int_{-s_{0}}^{s_{0}}\left| a\left(\frac{-t_{1}+s_{1}}{2\sqrt{t_{2}(t_{1}^{2} +1)}}\right) \right|   
		(h_{j-1}(s_{1}))^{2}d{s_{1}}\\
		&+& \int_{s_{0}}^{\infty} \left| a\left(\frac{-t_{1}+s_{1}}{2\sqrt{t_{2}(t_{1}^{2} +1)}}\right) \right|    
		(h_{j-1}(s_{1}))^{2} d{s_{1}}\\
		& < &   \frac{2 \epsilon}{3}  +
		\max_{-s_{0}< s_{1}< s_{0}}\left| a\left(\frac{-t_{1}+s_{1}}{2\sqrt{t_{2}(t_{1}^{2} +1)}}\right) \right|. 
	\end{eqnarray*}
		
	Because of the continuity of $a(s)$  at $-1/(2\sqrt{t_0})$, there exists $\delta_{1}>0$ such that
	$|a(s)|< \epsilon/3$ for $|s-\frac{-1}{2\sqrt{t_{0}}}|< \delta_{1}$. 
	Let us estimate the value of the  argument of $a$:
	\begin{eqnarray*}
		I&:=& \left|\frac{1}{2\sqrt{t_{2}(t_{1}^{2} +1)}}(-t_{1}+s_{1})- \frac{-1}{2\sqrt{t_{0}}} \right| \\
		& \leq&
		\left|- \frac{1}{2\sqrt{t_{2}}}+ \frac{1}{2\sqrt{t_{0}}} \right|\left| \frac{t_{1}}{\sqrt{t_{1}^{2} + 1}}  \right| 
		+\frac{1}{2\sqrt{t_{0}}} \left| 1- \frac{t_{1}}{\sqrt{t_{1}^{2} + 1}}  \right|  \\
		&+& \left| \frac{s_{1}}{2\sqrt{t_{2}(t_{1}^{2} + 1)}} \right|. % \\
%		& <& \delta_{1}
	\end{eqnarray*}
	Choose $\delta >0$ in such a way
	$\left|- \frac{1}{2\sqrt{t_{2}}}+ \frac{1}{2\sqrt{t_{0}}} \right| < \delta_{1} /3$ for $|t_{2}- t_{0}|<\delta$. 
	Pick $N_{1} >0$ such that
	$\left|1-\frac{t_{1}}{\sqrt{t_{1}^{2} + 1}}\right| < (2\sqrt{t_{0}}\delta_{1}) /3$  whenever $t_{1} > N_{1}$.
	Now assume that $|t_2-t_0|<\delta$ and $|s_{1}| < s_{0}$. Then $|\frac{1}{2\sqrt{t_{2}}}| < \frac{1}{2\sqrt{t_{0}}} + \frac{\delta_{1}}{3}$. 
	Thus,  $\frac{|s_{1}|}{2\sqrt{t_{2}(t_{1}^{2} + 1 )}}$ 
	converges to $0$ when $t_{1}$ tends to $+\infty$. Therefore, there exists $N>N_1$ such that
	$\frac{|s_{1}|}{2\sqrt{t_{2}(t_{1}^{2} + 1 )}}< \delta_{1} /3$ for	$t_{1}> N$.
	The additional condition $t_1>N$ implies
	$$\left| a\left(\frac{-t_{1}+s_{1}}{2\sqrt{t_{2}(t_{1}^{2} +1)}}\right) \right| < \epsilon /3.$$
	Hence $|\phi^{a}(t_{1}, t_{2})| < \epsilon$ if $|t_2-t_0|<\delta$ and $t_1>N$.
	
	If $a$  does not converge to zero at $-\frac{1}{2\sqrt{t_{0}}}$, then take the function
	$\tilde{a}(s)=a(s)-a_2$ and proceed as in the proof of Lemma \ref{lim-phi-infty-0}, where $a_2=a\left(-\frac{1}{2\sqrt{t_{0}}}\right)$. 
	
Finally, the justification of the limit  of $\phi^a(t_1,t_2)$ at $(-\infty,t_0)$ can be done analogously. 
\end{proof}

\begin{lem}\label{lim-phi-t0-infty}
	Let $a\in L^{\infty}(\R)$ be continuous at $0\in \R$. For $t_{0} \in \overline{\R}$, 
	the spectral function  $\phi^{a}$ satisfies
	$$\lim_{(t_{1}, t_{2}) \rightarrow (t_{0}, +\infty)} \phi^{a}(t_{1}, t_{2})= a(0).$$
	Actually, we have uniform convergence of $\phi^a(t_1,t_2)$ at $(t_0,+\infty)$, that is,  for $\epsilon>0$, 
	there exists $N>0$ such that $|\phi^{a}(t_{1}, t_{2})-a(0)|<\epsilon$ for all $t_2>N$ and for all $t_1\in \overline{\R}$. 
\end{lem}

\begin{proof}
	Suppose that $a(0)=0$. Let $\epsilon >0$, and choose 
	$s_{0}>0$ such that equations  (\ref{acotar-colas}) hold.
	Then
	\begin{eqnarray*}
		|\phi^{a}(t_{1}, t_{2})|
		& \leq & \int_{-\infty}^{-s_{0}} 
		\left| a\left(\frac{-t_{1}+s_{1}}{2\sqrt{t_{2}(t_{1}^{2} +1)}}\right) \right| (h_{j-1}(s_{1}))^{2}  d{s_{1}} \\
		&+& \int_{-s_{0}}^{s_{0}}\left| a\left(\frac{-t_{1}+s_{1}}{2\sqrt{t_{2}(t_{1}^{2} +1)}}\right) \right|   
		(h_{j-1}(s_{1}))^{2}d{s_{1}}\\
		&+& \int_{s_{0}}^{\infty} \left| a\left(\frac{-t_{1}+s_{1}}{2\sqrt{t_{2}(t_{1}^{2} +1)}}\right) \right|    
		(h_{j-1}(s_{1}))^{2} d{s_{1}}\\
		& < &    \frac{2 \epsilon}{3}  +
		\max_{-s_{0}< s_{1}< s_{0}}\left| a\left(\frac{-t_{1}+s_{1}}{2\sqrt{t_{2}(t_{1}^{2} +1)}}\right) \right|.
	\end{eqnarray*}
	
	By the continuity  of $a(s)$ at $0$, there exists $\delta_{1}>0$ such that $|a(s)|< \epsilon/3$ for 
	$|s| < \delta_{1}$. For  $ -s_0 <s_1 <s_0$ we have
	\begin{eqnarray*}
		\left| \frac{-t_{1}+s_{1}}{2\sqrt{t_{2}(t_{1}^{2} +1)}} \right| &\leq& 
		\frac{1}{2\sqrt{t_{2}}} \left( \left|\frac{t_{1}}{\sqrt{t_{1}^{2} + 1}}\right| +\frac{|s_{1}|}{\sqrt{t_{1}^{2} + 1}} \right)
		 <   \frac{1}{2\sqrt{t_{2}}}(1+s_0).
	\end{eqnarray*}

	Take $N=(1+s_{0})^{2}/(4 \delta_{1}^{2})$. The inequality  $t_2>N$ implies
	$\frac{1}{2\sqrt{t_2}}<\frac{\delta_{1}}{(1+s_0)}$. Thus, 
  if $t_{2} >N$, $t_{1} \in \overline{\R}$ and $-s_{0}<s_1 < s_{0}$, then
	$$\left| \frac{-t_{1}+s_{1}}{2\sqrt{t_{2}(t_{1}^{2} +1)}} \right| < \delta_{1}.$$
Consequently, 	$|\phi^{a}(t_{1}, t_{2})|<\epsilon$ for all $t_2>N$ and $t_1 \in \overline{\R}$.
	
	Finally, in the case $a(0) \neq 0$, the proof can be carry out by considering the symbol  $\tilde{a}(s)=a(s)-a(0)$.
\end{proof}

For each nilpotent symbol $a(\im \zeta_1) \in C(\overline{\R})$, the spectral function $\phi^{a}$ is continuous on $\overline{\Pi}$ 
and is constant along $\overline{\R}\times \{+\infty\}$.
For this reason, the $C^*$-algebra generated by all spectral functions $\phi^{a}$ is not $C(\overline{\Pi})$, 
but it coincides with the algebra of continuous functions on a quotient space of $\overline{\Pi}$.

%%%%%%%%%%%%%%%%%%%%%%%%%%%%%%%%%%%%%%%%%%%%%%%%%%%%

\subsection{Toeplitz operators $T_a$ with continuous symbols $a(\text{Im}\, \zeta_1)$.}

Introduce the quotient space $\triangle:=\overline{\Pi} / (\overline{\R} \times \{+\infty\})$. 
By Lemma \ref{lim-phi-t0-infty}, the spectral function $\phi^{a}$ can be identified with a continuous function on $\triangle$, 
which will be also denoted by $\phi^{a}: \triangle \rightarrow \C$.  
We establish now one of our main results in this work.

\begin{thm}
		The $C^{*}$-algebra generated by all spectral functions
		$\phi^{a}$, with  $a(\im \zeta_1) \in C(\overline{\R})$,  is isomorphic and isometric to the  algebra  $C(\triangle)$. 
		That is,  the $C^*$-algebra generated by all Toeplitz  operators $T_a$ is isomorphic 
	to $C(\triangle)$, where the isomorphism is defined on the generators by the rule
	$$T_a  \longmapsto \phi^a.$$
\end{thm}

\begin{proof}
The functions $\phi^{a}$ separate the points of $\triangle$ according to
Lemmas \ref{separar1}, \ref{separar2}, \ref{separar3} below.  
The Stone-Weierstrass theorem completes the proof.
\end{proof}

\begin{lem}\label{separar1}
Let $(t_1, t_2)$ and $(y_1, y_2)$ be distinct points of $\overline{\Pi}\setminus \Pi$,
where they do not belong simultaneously to $\overline{\R} \times \{+\infty \}$. 
Then there exists $a \in C(\overline{\R})$ such that
$\phi^{a}(t_1, t_2) \neq \phi^{a}(y_1, y_2)$.
\end{lem}

\begin{proof}
First consider the nilpotent symbol  $a_{1}(s)= \frac{s}{\sqrt{s^{2}+1}}$, which is continuous on $\overline{\R}$. 
Note that
\begin{itemize}
	\item $\phi^{a_{1}}(\pm\infty, t_{2})=a_{1}\left(\mp\frac{1}{2\sqrt{t_{2}}}\right)= \mp \frac{1}{\sqrt{1+4t_{2}}}$ for $0<t_{2} \leq +\infty$.
	\item $\phi^{a_{1}}(\pm \infty, 0)=a_{1}(\mp \infty)=  \mp 1$.
	\item $\phi^{a_{1}}(t_1, +\infty)= a_{1}(0)= 0$ for $t_1\in \overline{\R}$.
	\item $\phi^{a_{1}}(t_{1}, 0)= -\int_{-\infty}^{t_{1}} (h_{j-1}(s_{1}))^{2} ds_{1}
	+   \int_{t_{1}}^{\infty} (h_{j-1}(s_{1}))^{2} ds_{1}$ for $t_{1} \in \R$.
\end{itemize}

From $\int_{-\infty}^{\infty} (h_{j-1}(s_{1}))^{2} ds_{1}=1$ we get
$$\phi^{a_{1}}(t_{1}, 0)  =  2 \int_{t_{1}}^{\infty} (h_{j-1}(s_{1}))^{2} ds_{1} - 1.$$
This formula also says that $\phi^{a_{1}}(\pm \infty, 0)=  \mp 1$. On the other hand, the Hermite function $h_{j-1}$
is continuous and it has just a finitely many roots. Hence, $\phi^{a_{1}}(t_{1}, 0)$ is monotonically decreasing
with respect to $t_1$. Thus, points in $\overline{\R} \times \{0\}$ are separated by $\phi^{a_{1}}$.

Recall that $\overline{\R} \times \{+\infty\}$ is identified with one point in $\triangle$,
so take $(-\infty,+\infty)$ as representative point of the equivalence class $\overline{\R} \times \{+\infty\}$.
For $t_2,\tilde{t}_2\in [0,\infty)$, the three points $(-\infty, t_2)$, $(-\infty, +\infty)$ and $(+\infty,\tilde{t}_2)$ 
are separated by $\phi^{a_{1}}(\pm\infty, t)$  because of the injective property of $1/\sqrt{1+4t}$.
 
Consider now the nilpotent symbol $a_{2}(s)=\frac{1}{s^{2} +1}$, which is continuous on $\overline{\R}$. 
We have
\begin{itemize}
	\item $\phi^{a_{2}}(\pm\infty, t_{2})=a_{2}\left(\mp\frac{1}{2\sqrt{t_{2}}}\right)= \frac{4t_{2}}{1+4t_{2}} >0$
	for $t_{2} >0$.
	\item $\phi^{a_{2}}(t_{1}, +\infty)= a_{2}(0)= 1$ for $t_1\in \overline{\R}$.
	\item $\phi^{a_{2}}(t_{1}, 0)=0$ for all $t_{1} \in \overline{\R}$.
\end{itemize}
Thus, $\phi^{a_{2}}$ separates each point $(t_{1}, 0)$ from the points $(\pm\infty, t_{2})$ and $(t_{1}, +\infty)$.
%As\'i quedan separados todos los puntos del contorno de $\overline{\Pi}$. 
\end{proof}

Now our aim is  to separate the points of $\Pi$. 
Consider the following family of continuous functions:
\begin{equation}\label{a-alpha}
a_{\alpha}(s)= \left\{ 
\begin{array}{lcc}
0 &  \text{si}  & -\infty \leq s\leq -\alpha, \\
\frac{1}{\alpha}s + 1 &  \text{si} & -\alpha \leq s\leq  0,\\
1 &   \text{si}  & 0 \leq s\leq +\infty.
\end{array}
\right.
\end{equation}
Then $\phi^{a_{\alpha}}(t_{1}, t_{2})=\psi^{\alpha}(t_1,t_2)+\varphi(t_1)$, where 

$$\psi^{\alpha}(t_1,t_2)=\int_{t_1-2\alpha\sqrt{t_{2}(t_{1}^{2}+ 1)}}^{t_{1}} 
\left(\frac{s_{1}- t_{1}}{2\alpha\sqrt{t_{2}(t_{1}^{2}+ 1)}} +1 \right)(h_{j-1}(s_{1}))^{2} ds_{1},$$
\begin{equation}\label{part-espectral}
\varphi(t_1)= \int_{t_{1}}^{\infty} (h_{j-1}(s_{1}))^{2} ds_{1}.
\end{equation}

\begin{lem}\label{separar2}
	If $(t_1,t_2)$ and $(y_1,y_2)$ are distinct points in $\Pi$, then there exists $\alpha>0$ such that
	$\phi^{a_\alpha}(t_1,t_2) \neq \phi^{a_\alpha}(y_1,y_2)$, where $a_{\alpha}$ is defined in (\ref{a-alpha}).
\end{lem}

\begin{proof}
At first suppose that $(t_{1}, t_{2})$ and $(y_{1}, y_{2})$ satisfy  $t_{1} < y_{1}$. 
Introduce $k:=\varphi(t_1)-\varphi(y_1)>0$.
It is easy to see that $\psi^{\alpha}(t_1,t_2)$ can be written as
$$\psi^{\alpha}(t_1,t_2)=\int_{-1}^{0}\left(s +1 \right)[h_{j-1}(t_1+2\alpha \sqrt{t_2(t_1^2+1)}s)]^{2} 2\alpha\sqrt{t_2(t_1^2+1)} ds.$$
This integral representation allow us to prove easily  that
$$\lim_{\alpha\rightarrow 0^+} \psi^{\alpha}(t_1,t_2)=0.$$

Take $\alpha$ small enough that $|\psi^{\alpha}(t_1,t_2)-\psi^{\alpha}(y_1,y_2)|<k/2$.
Then
\begin{eqnarray*}
	|\phi^{a_{\alpha}}(t_{1}, t_{2})-\phi^{a_{\alpha}}(y_{1}, y_{2})|
	&=& |k+\psi^{\alpha}(t_{1}, t_{2})-\psi^{\alpha}(y_{1}, y_{2})| \\
	& \geq & |k-|\psi^{\alpha}(t_{1}, t_{2})-\psi^{\alpha}(y_{1}, y_{2})| | \\
	&\geq & k- \frac{k}{2}>0.
	%&=& \frac{k}{2}.
\end{eqnarray*}

Now suppose that $t_1=y_1=t$ and $t_2<y_2$. Let $\alpha = \frac{1}{2\sqrt{t^2+1}} >0$. Then
$$\psi^{\alpha}(t,t_2)=\int_{t-\sqrt{t_{2}}}^{t} 
\left(\frac{s_{1}- t}{\sqrt{t_{2}}} +1 \right)(h_{j-1}(s_{1}))^{2} ds_{1},$$
$$\psi^{\alpha}(t,y_2)=\int_{t-\sqrt{y_{2}}}^{t} 
\left(\frac{s_{1}- t}{\sqrt{y_{2}}} +1 \right)(h_{j-1}(s_{1}))^{2} ds_{1}.$$

Since $0<t_2 < y_2$, we have $t-\sqrt{y_2} < t- \sqrt{t_2}$. Besides, 
the following inequality holds
$$0<\frac{s_1 -t}{\sqrt{t_2}} +1  <\frac{s_1 -t}{\sqrt{y_2}} +1$$
for  $t-\sqrt{t_2} < s_1 < t$.
Hence $\psi^{\alpha}(t,t_2) < \psi^{\alpha}(t,y_2)$. Finally, from
$\varphi(t_1)= \varphi(y_1)= \varphi(t)$
we get $\phi^{a_{\alpha}}(t, t_{2}) < \phi^{a_{\alpha}}(t, y_{2})$.
\end{proof}

\begin{lem}\label{separar3}
Let $(t_1, t_2) \in \Pi$ and $(y_1,y_2)\in \overline{\Pi}\setminus \Pi$. Then there exists $\alpha>0$ such that
$\phi^{a_\alpha}(t_1,t_2) \neq \phi^{a_\alpha}(y_1,y_2)$, where $a_{\alpha}$
is given in (\ref{a-alpha}).
\end{lem}

\begin{proof}
Take $(t_{1}, t_{2})\in \Pi$, $(y_{1}, 0)\in \mathbb{R}\times \{0 \}$ and $\alpha >0$. It is easy to see that  
 $\psi^{a_\alpha}(t_1,t_2) >0$. 
 Then we have
$\phi^{a_{\alpha}}(y_{1},0)=\varphi(y_1) \leq \varphi(t_1) < \phi^{a_{\alpha}}(t_{1}, t_{2})$ for $t_1 \leq y_1$.
Now suppose that $y_{1}< t_{1}$. Let $k:=\varphi(y_1)-\varphi(t_1)>0$. The inequality  
$\psi^{\alpha}(t_1,t_2) < k$ holds for $\alpha>0$ small enough. For such $\alpha$,
$$ 	\phi^{a_{\alpha}}(t_{1}, t_{2})=\psi^{\alpha}(t_1,t_2) + \varphi(t_1) < \varphi(y_1)
=  \phi^{a_{\alpha}}(y_{1},0).$$
This proves that all points of $\Pi$ can be separated from points of $\R \times \{0 \}$.
On the other hand,
\begin{eqnarray*}
	0<\phi^{a_{\alpha}}(t_{1}, t_{2})
	&=& \psi^{\alpha}(t_1,t_2)+\varphi(t_1)\\
	&\leq&  \int_{t_1-2\alpha\sqrt{t_{2}(t_{1}^{2}+ 1)}}^{t_{1}} (h_{j-1}(s_{1}))^{2} ds_{1} 
	+ \varphi(t_1)\\
	&=&  \int_{t_1-2\alpha\sqrt{t_{2}(t_{1}^{2}+ 1)} }^{\infty} (h_{j-1}(s_{1}))^{2} ds_{1}\\
	&<& 1 .
\end{eqnarray*}

For $y_{2} \in \overline{\R}_+$ we have
$\phi^{a_{\alpha}}(- \infty, y_{2}) = a_{\alpha}(\frac{1}{2\sqrt{y_{2}}})= 1$,  hence 
$\phi^{a_{\alpha}}(t_{1}, t_{2}) \neq \phi^{a_{\alpha}}(- \infty, y_{2}).$

Finally, consider the function $a(s):=a_{\alpha}(s-\alpha)$. 
Then $\phi^{a}(+\infty, y_{2})=a(-\frac{1}{2\sqrt{y_2}})=0$, and 
consequently $\phi^{a}(t_{1}, t_{2}) \neq \phi^{a}(+\infty, y_{2}) $.
\end{proof}

%%%%%%%%%%%%%%%%%%%%%%%%%%%%%%%%%%%%%%%%%%%%%%%%%%%%
%%%%%%%%%%%%%%%%%%%%%%%%%%%%%%%%%%%%%%%%%%%%%%%%%%%%

\section{Toeplitz operators with piecewise-continuous symbols $a(\text{Im}\, \zeta_1)$.}\label{symb-discont}

Take a finite subset $S=\{\beta_0,...,\beta_m \}\subset \R$  and let $PC(\overline{\R},S)$ be the set of 
functions continuous on $\overline{\R}\setminus S$  and having one-side limit values at each point of $S$.
In this section we study  the $C^*$-algebra generated by all Toeplitz operators $T_a$, where $a(\im \zeta_1)\in PC(\overline{\R},S)$.
Obviously, we have to study the algebra generated by the spectral functions $\phi^a$.
To begin with, take the indicator function $\chi_+=\chi_{[0, +\infty]}$. Then
the spectral function $\phi^{\chi_+}$ is continuous on $\overline{\Pi}\setminus (\overline{\R}\times \{+\infty \})$
according to Lemmas  \ref{lim-gamma-x0-0},  \ref{lim-phi-infty-0} and \ref{lim-phi-infty-t0}.
Actually,
\begin{eqnarray*}
	\phi^{\chi_+}(t_1, t_2) &=& \int_{-\infty}^{\infty} \chi_{[0, +\infty]}\left(\frac{-t_{1}+s_{1}}{2\sqrt{t_{2}(t_{1}^{2} +1)}}\right) (h_{k-1}(s_1))^2 ds_1\\
	&=&  \varphi(t_1),
\end{eqnarray*}
where $\varphi$ is defined in (\ref{part-espectral}). 
Hence $\phi^{\chi_+}$ is continuous on $\overline{\Pi}$ because $\varphi$ is.
Of course we have now a spectral function which is not constant along $\overline{\R}\times \{+\infty \}$ anymore.

For any  $a \in PC(\overline{\R},\{0 \})$ we have
$$a(s)=\tilde{a}(s)+[a(0_+)-a(0_-)]\chi_+(s),$$
where $a(0_-)$ and $a(0_+)$ are the one-side limits of $a$ at $0$, and $\tilde{a}(s)= a(s)+[a(0_-)-a(0_+)]\chi_+(s)$.
This function has a removable discontinuity at $0$, thus
$\phi^a$ is continuous on $\overline{\Pi}$.

\begin{thm}\label{Topelitz-0}
	The $C^*$ algebra generated by all Toeplitz operators $T_a$, with $a(\im \zeta_1) \in PC(\overline{\R}, \{0\})$, 
	is isomorphic and isometric  to $C(\overline{\Pi})$. The isomorphism is defined on the generators by the rule
	$$T_a \mapsto \phi^a.$$
\end{thm}

\begin{proof}
	If  $t_1\neq y_1$, the function $\phi^{\chi_+}$ separates any two points $(t_1,t_2), \ (y_1,y_2) \in \overline{\Pi}$.
	By Lemma \ref{separar2}, two points in $\Pi$ can be separated by  $\phi^{a_{\alpha}}$, where 
	$a_{\alpha}$ is the nilpotent symbol given in (\ref{a-alpha}). Further,
	$$\phi^{a_{\alpha}}(t_1,0)=\varphi(t_1)< \phi^{a_{\alpha}}(t_1,t_2)<1=a_{\alpha}(0)=\phi^{a_{\alpha}}(t_1,+\infty).$$
%	Thus, $\overline{\Pi}$.
\end{proof}

We continue our study by introducing another point of discontinuity. Take the indicator function
$\chi_\beta:=\chi_{[\beta/2, +\infty] }(s)$, where $\beta >0$. We have
\begin{eqnarray*}
	\phi^{\chi_\beta}(t_1, t_2) 
	&=& \int_{-\infty}^{\infty} \chi_{[\beta/2, +\infty]}\left(\frac{-t_{1}+s_{1}}{2\sqrt{t_{2}(t_{1}^{2} +1)}}\right) (h_{k-1}(s_1))^2 ds_1\\
	&=&  \int_{t_1+\beta \sqrt{t_{2}(t_{1}^{2}+ 1)}}^{\infty} (h_{k-1}(s_1))^2 ds_1.
\end{eqnarray*}
According to Lemmas \ref{lim-gamma-x0-0},  \ref{lim-phi-infty-0}, \ref{lim-phi-infty-t0} and \ref{lim-phi-t0-infty}, 
the spectral function $\phi^{\chi_\beta}$ is continuous on $\overline{\Pi}$, except at the point $(-\infty,1/\beta^2)$.
For $\lambda_0 \in \mathbb{R}$,  $\phi^{\chi_\beta}$ takes the constant value $\varphi(\lambda_0)$ along
the curve $t_1+\beta \sqrt{t_{2}(t_{1}^{2}+ 1)}=\lambda_0$. From this equation we get
\begin{equation}\label{level-chi-beta}
t_2=f(t_1)=\frac{(t_1-\lambda_0)^2}{\beta^2(t_1^2+1)}, \qquad t_1\leq \lambda_0.
\end{equation}
The horizontal line $t_2=1/\beta^2$ is an asymptote of the graph of $t_2=f(t_1)$,
 and of course $\phi^{\chi_\beta} (t_1, f(t_1)) = \varphi(\lambda_0)$ for $t_1\leq \lambda_0.$
Thus, the level curves of $\phi^{\chi_\beta}$ converge to the point $(-\infty,1/\beta^2)$, 
our aim is to separate them at $(-\infty,º/\beta^2)$  through a mapping $\Upsilon_{\beta}$ in such a way that 
$\phi^{\chi_{\beta}} \circ \Upsilon_{\beta}^{-1}$ is continuous on $\overline{\Pi}$.

\begin{lem}\label{Phi-2}
	Let $g: \overline{\R}\rightarrow [-1/2,1/2]$ be any bijective, smooth and increasing function,
	with $g(0)=0$.
	Take $\beta>0$, and let $\Upsilon_{\beta}$ be  the function on $\Pi$ defined by the rule
	\begin{equation}\label{Phi-beta}
		\Upsilon_{\beta}(t_1,t_2)=(t_1,t_2[1+g(\lambda(t_1,t_2))]),
	\end{equation}
	where $\lambda(t_1,t_2)=t_1+\beta\sqrt{t_2(t_1^2+1)}$.
	Then $\Upsilon_{\beta}$ is an  homeomorphism from $\Pi$ onto itself, which can be continuously extended to
	$\overline{\Pi} \setminus \{(-\infty, 1/\beta^2)\}$ with range 
	$\overline{\Pi}\setminus (\{-\infty\} \times J_{\beta})$, where $J_{\beta}=\{y \ : \  .5/\beta^2 \leq y \leq 1.5/\beta^2 \}$.
\end{lem} 

\begin{proof}
	The function $h(t_1,t_2):=1+g(\lambda(t_1,t_2))$ has range contained in $[1/2,3/2]$. 
	Hence $\Upsilon_{\beta}(\Pi) \subset \Pi$. Now suppose that $\Upsilon_{\beta}(t_1,t_2)=\Upsilon_{\beta}(a,b)$.
	Then $t_1=a$ and $t_2h(a,t_2)=b h(a,b)$. The function $h(a,t_2)$
	is strictly increasing with respect to $t_2$. Thus, $t_2>b$ implies that $t_2h(a,t_2)>b h(a,b)$.
	Consequently, $\Upsilon_{\beta}$ is injective. Let $(c,d)$ be a point in $\Pi$.
	Consider the  equation $\Upsilon_{\beta}(t_1,t_2)=(c,d)$, which is equivalent to the system of equations
	$t_1=c$, $t_2h(t_1,t_2)=d$. Thus, we have to prove that $t_2h(c,t_2)=d$ is solvable.
The function $t_2h(c,t_2)$ is bijective 
	from $(0,+\infty)$ into itself 	with respect to $t_2$. Therefore $t_2h(c,t_2)=d$ has a unique solution.
	That is, $\Upsilon_{\beta}:\Pi \rightarrow \Pi$ is surjective. 
	On the other hand, $\Upsilon_{\beta}$ is smooth, and so is $\Upsilon_{\beta}^{-1}$
	because of the Inverse Function Theorem.
	
	The correspondence (\ref{Phi-beta}) also defines $\Upsilon_{\beta}$ on $\mathbb{R} \times \{ 0\}$ by
	$\Upsilon_{\beta}(t_1,0)=(t_1,0)$.  Actually, $\Upsilon_{\beta}$ can be defined on
	$\overline{\Pi}\setminus (\Pi \cup \{ (-\infty,1/\beta^2) \} )$ according to the 
	following limits:
	\begin{eqnarray*}
	\lim_{(t_{1}, t_{2}) \rightarrow (t_1^0, 0)} \Upsilon_{\beta}(t_1,t_2)   &=& (t_1^0,0), \quad t_1^0 \in \overline{\R}, \\
	\lim_{(t_{1}, t_{2}) \rightarrow (+\infty, t_2^0)} \Upsilon_{\beta}(t_1,t_2) &=& (+\infty, \frac{3}{2} t_2^0), \quad t_2^0 \in \overline{\R}, \\
	\lim_{(t_{1}, t_{2}) \rightarrow (t_1^0, +\infty)} \Upsilon_{\beta}(t_1,t_2) &=& (t_1^0,+\infty), \quad t_1^0 \in \overline{\R}, \\
	\lim_{(t_{1}, t_{2}) \rightarrow (-\infty,t_2^0)} \Upsilon_{\beta}(t_1,t_2) &=& (-\infty, \frac{1}{2} t_2^0), \quad 0<t_2^0<1/\beta^2, \\
	\lim_{(t_{1}, t_{2}) \rightarrow (-\infty,t_2^0)} \Upsilon_{\beta}(t_1,t_2) &=& (-\infty, \frac{3}{2} t_2^0), \quad t_2^0>1/\beta^2.
	\end{eqnarray*}
We will justify just the last limit. For $t_1<0$,
$$\lambda(t_1,t_2)=|t_1| \left[ -1+ \beta \sqrt{ t_2 \left( 1+\frac{1}{t_1^2} \right) } \right].$$
If $t_2^0>1/\beta^2$ and $t_2$ is close enough to $t_2^0$, then $\lambda(t_1,t_2)$ tends to $+\infty$ when $t_1$ tends to $-\infty$.
Thus, $\lim\limits_{(t_1,t_2) \rightarrow (-\infty,t_2^0)} g(t_1,t_2)=1/2$. Consequently, 
$$\lim\limits_{(t_1,t_2) \rightarrow (-\infty,t_2^0)} t_2h(t_1,t_2)=\frac{3}{2} t_2^0,$$
that is, $\Upsilon_{\beta}(t_1,t_2) $ converges to $(-\infty, 3t_2^0/2)$ when $(t_1,t_2)$ tends to $ (-\infty,t_2^0)$.
A local analysis proves that $\Upsilon_{\beta}^{-1}$ is continuous.
This completes the proof.
\end{proof}

\begin{lem}\label{cont-phi-phi-beta}
The function $\phi^{\chi_{\beta}} \circ \Upsilon_{\beta}^{-1}$ is continuous on $\overline{\Pi}$ and 
it separates the points in the line segment $\{-\infty\} \times J_{\beta}$.
For each  continuous function $\phi:\overline{\Pi} \rightarrow \C$,  $\phi \circ \Upsilon_{\beta}^{-1}$  is also 
continuous on $\overline{\Pi}$ and has constant value along $\{-\infty\} \times J_{\beta}$.
\end{lem}

With a discontinuity $\beta_1=\beta$, we define $\Phi_1=\Upsilon_{\beta}$.
Introduce another point of discontinuity $\beta_2$, with $\beta_1<\beta_2$. 
Let $P_2=(-\infty, 0.5/\beta_2^2)$.
The function $\phi^{\chi_{\beta_2}} \circ \Phi_1^{-1}$ has a continuous extension to
$\overline{\Pi} \setminus \{P_2\}$, and its level curves 
$\Phi_1\left(t_1, \frac{(t_1-\lambda)^2}{\beta_2^2(t_1^2+1)}\right)$, $t_1>\lambda$, converge to $P_2$. 
As in Lemma \ref{Phi-2}, we can construct a mapping  $\Phi_{2}$ that separates all these level curves, and  
$\phi^{\chi_{\beta_2}} \circ \Phi_1^{-1} \circ \Phi_2^{-1}$ is continuous on  $\overline{\Pi}$. 
Adding more discontinuity points $\beta_3,...,\beta_m$, with $\beta_j < \beta_{j+1}$, we can construct mappings $\Phi_3,..., \Phi_m$ in such a way that $\phi^{\chi_{\beta_m}} \circ \Phi_{1}^{-1} \circ \cdots \circ \Phi_{m}^{-1}$ is continuous on  $\overline{\Pi}$.

Let ${\cal T}_{S_m}$ denote the $C^*$-algebra generated by all Toeplitz operators $T_a$
with nilpotent symbols $a(\text{Im}\, \zeta_1) \in PC(\overline{\R},S_m)$,
where $S_m=\{\beta_0,...,\beta_m \}$. For simplicity in our explanation, we assume that $\beta_0=0$
and $S_{m-1} \subset S_m \subset \R_+\cup \{0 \}$. We will explain how the Toeplitz algebra 
${\cal T}_{S_m}$ increases as $S_m$ does. By Theorem \ref{Topelitz-0}, ${\cal T}_{S_0}$ is isomorphic to
$C(\overline{\Pi})$, where the isomorphism is given on the generators by the rule
$$\Psi_0: {\cal T}_{S_0} \ni T_a \mapsto \phi^a \in C(\overline{\Pi}).$$
Consider now the  algebra ${\cal T}_{S_1}$, where $S_1=\{0,\beta_1 \}$.
By Lemma \ref{cont-phi-phi-beta}, $\phi^a \circ \Phi_{1}^{-1}$ is continuous on $\overline{\Pi}$
for every $a\in PC(\overline{\R},S_1)$. Then the algebra ${\cal T}_{S_1}$ is also
isomorphic and isometric to $C(\overline{\Pi})$, where the isomorphism is given by
$$\Psi_1: {\cal T}_{S_1} \ni T_a \mapsto \phi^a\circ \Phi_{1}^{-1} \in C(\overline{\Pi}).$$
At first sight, both algebras ${\cal T}_{S_0}$ and ${\cal T}_{S_1}$ seem to have the same spectrum $\overline{\Pi}$ 
but they do not, they are identified with $C(\overline{\Pi})$ through different isomorphisms. 
Of course, ${\cal T}_{S_0}$ is a subalgebra of ${\cal T}_{S_1}$. 
If $T_a \in {\cal T}_{S_0}$, then $\phi^a\circ \Phi_{1}^{-1} \in C(\overline{\Pi})$ and has
constant value along the line segment $I_1=\{-\infty\} \times J_{\beta_1}$.
According to the isomorphism $\Psi_1$, we can say that the spectrum of ${\cal T}_{S_1}$
equals $\overline{\Pi}$ meanwhile the spectrum of ${\cal T}_{S_1}$ is the quotient
space $\overline{\Pi}/I_{1}$. This phenomenon persists as long as the set  $S_m$ grows.

\begin{thm}\label{Topelitz-m}
	Let $S_m=\{0,\beta_1,...,\beta_m \} \subset \R_+$.
	Then there exist bijective continuous functions $\Phi_j:\Pi \rightarrow \Pi$, $j=1,...,m$, such that
	$\phi^a \circ \Phi_{1}^{-1} \circ \cdots \circ \Phi_{m}^{-1}$ admits a continuous extension to
	$\overline{\Pi}$ for each $a(\im \zeta_1) \in PC(\overline{\R}, S_m)$.
	The $C^*$ algebra generated by all Toeplitz operators $T_a$ 
	is isomorphic and isometric  to $C(\overline{\Pi})$. The isomorphism is defined on the generators by the rule
	$$\Psi_m :  T_a \mapsto \phi^a \circ \Phi_{1}^{-1} \circ \cdots \circ \Phi_{m}^{-1}.$$
\end{thm}

Note that for each piece-wise continuous symbol $a(\im \zeta_1) \in PC(\overline{\R}, S_m)$, in general 
the spectral function $\gamma^a:\Pi \rightarrow \mathbb{C}$ 
does not admit a continuous extension to $\overline{\Pi}$ but 
$\gamma^a \circ \Phi^{-1} \circ \Phi_{1}^{-1} \circ \cdots \circ \Phi_{m}^{-1}$
does, which means that $\gamma^a$ is uniformly continuous with respect to 
a new metric on $\Pi$, this metric is the
pushforward of the usual metric  using  the mapping $\Phi^{-1} \circ \Phi_{1}^{-1} \circ \cdots \circ \Phi_{m}^{-1}$.

\section{Toeplitz operators with symbols $a(\im\zeta_1) b(\im \zeta_2 - |\zeta_1|^2)$.}\label{T-ab}

In this section,  we describe the $C^*$-algebra generated by all Toeplitz operators
with symbols of the form
$c(\zeta_1,\zeta_2)=a(\im\zeta_1) b(\im \zeta_2 - |\zeta_1|^2)$,
where $a(s) \in C(\overline{\R})$, and $b(t)\in L^{\infty}(\R_+)$ has limits
values at $t=0,+\infty$. For such a symbol $c$ we have
that $\gamma^c=\gamma^a \gamma^b$, which means that $T_c=T_a T_b=T_b T_a$.
Although $\gamma^b$ belongs to $C(\overline{\Pi})$, the spectral function 
$$(\gamma^b \circ \Phi^{-1})(t_1,t_2)
=\int_{\R_{+}} b\left(\frac{y_2}{2(t_1^2+1)t_2} \right) (\ell_{k-1}(y_2))^2 dy_2$$
is continuous on $\overline{\Pi}\setminus \{P_-,P_+\}$, where
$P_-=(-\infty,0)$ and $P_+=(+\infty,0)$. Since the level curves of $\gamma^b(x_1,x_2)$ are the 
horizontal lines $x_2=\mu$, the level curves of $\gamma^b \circ \Phi^{-1}$
are given by the equations $t_2=\mu/(t_1^2+1)$, with $\mu\in \R_+$.

\begin{lem}
	Let $f: [0,+\infty] \rightarrow [0,1]$ be any bijective, smooth and increasing function.
	Then the function 
	\begin{equation}\label{Psi}
	\Theta(t_1,t_2)=\left(t_1,t_2+\frac{t_1^2}{t_1^2+1}f(t_2[t_1^2+1])\right),
	\end{equation}
	is an  homeomorphism from $\Pi$ onto itself, which can be continuously extended to
	$\overline{\Pi} \setminus \{(P_-,P_+\}$ with range 
	$\overline{\Pi}\setminus I_{\infty}$, where 
	$I_{\infty}=\{ (\tau_1,\tau_2) : \ \tau_1=\pm\infty \ \text{and} \   0 \leq \tau_2 \leq 1 \}$.
	We have 
	$\Theta(\pm \infty,t_2)=(\pm\infty,t_2+1)$ for $0<t_2< +\infty,$
	and $\Theta$ acts like the identity mapping at the rest of points in 
	$\overline{\Pi}\setminus (\Pi \cup \{P_-,P_+\})$. 
\end{lem} 

\begin{proof}
	Similar to the proof of Lemma \ref{Phi-2}.
\end{proof}

The image of the level curve $t_2=\mu/(1+t_1^2)$ under $\Theta$ is the curve
$$\tau_2=\frac{\mu}{\tau_1^2+1} + f(\mu) \frac{\tau_1^2}{\tau_1^2+1}.$$
This means that the level curves of $\gamma^b \circ \Phi^{-1} \circ \Theta^{-1}$
do not converge to a single point anymore.

\begin{lem}\label{cont-gamma-b-psi}
	The function $\gamma^{b} \circ \Phi^{-1} \circ \Theta^{-1}$ is continuous on $\overline{\Pi}$,
	and for each continuous function $\phi:\overline{\Pi} \rightarrow \C$,  $\phi \circ \Theta^{-1}$  is also 
	continuous on $\overline{\Pi}$ and has constant value along each component of $I_{\infty}$.
\end{lem}

\begin{thm}%\label{Top}
	The $C^*$ algebra generated by all Toeplitz operators $T_{ab}$, with $a(\im \zeta_1) \in PC(\overline{\R}, \{0\})$ and $b(t)\in L^{\infty}(\R_+)$ having limits
	values at $t=0,+\infty$, 
	is isomorphic and isometric  to $C(\overline{\Pi})$. The isomorphism is defined on the generators by the rule
	$$T_{ab} \mapsto \gamma^{ab} \circ \Phi^{-1} \circ \Theta^{-1}.$$
\end{thm}

For the Toeplitz $T_{\chi_\beta}$ with symbol 
$\chi_{\beta}=\chi_{[\beta/2,+\infty)}(\im \zeta_1)$,
there exists a mapping $\Theta_{\beta}:\Pi \rightarrow \Pi$ such that
$\gamma^{\chi_\beta} \circ \Phi^{-1} \circ \Theta^{-1} \circ \Theta_{\beta}^{-1}$
admits a continuous extension to $\overline{\Pi}$.
The construction of $\Theta_{\beta}$ is similar to the construction of $\Phi_{\beta}$
given in Lemma \ref{Phi-2}, where one has to take into account the level curves 
of the spectral function 
$\gamma^{\chi_\beta} \circ \Phi^{-1} \circ \Theta^{-1}$, which
converge to the point $(-\infty,1+1/\beta^2)$.

Implicitly, we have considered several compactifications of $\Pi$
associated to the $C^*$-algebras studied here in, each compactification depends on the kind of symbols.
Take $Q_-=(-\infty,+\infty)$ and $Q_+=(+\infty,+\infty)$,
let us explain the situation in the case of the algebra generated by the Toeplitz 
with symbols $a(\im \zeta_1) \in C(\overline{\R})$.
Essentially the corresponding compactification of $\Pi$ is obtained from
$\overline{\Pi}\setminus (\{Q-_,Q_+\})$ by gluing a line segment 
at each corner $Q_-$ and $Q_+$.
Each spectral function $\gamma^a(x_1,x_2)$ is continuous on 
$\overline{\Pi}\setminus (\{Q-_,Q_+\})$, and has limit values when $(x_1,x_2)$ moves 
along the parabolas $x_2=\mu+\mu x_1^2$ and tends to $Q_{\pm}$.
For a net $\{X_{\lambda}\}$ tending to $Q_+$, $\{\gamma^a(X_{\lambda})\}$
converges if $\{X_{\lambda}\}$ is eventually in gaps between two parabolas 
close enough from each other with respect to the parameter $\mu$.

\vspace{.4cm}
\noindent
{\bf Acknowledgments}

\vspace{.3cm}
\noindent
This work was supported by Universidad Veracruzana 
under P/PROFEXCE-2020-30MSU0940B-22 project, M\'exico.

%%%%%%%%%%%%%%%%%%%%%%%%%%%%%%%%%%%%%%%%%%%%%%%%%%%%
%%%%%%%%%%%%%%%%%%%%%%%%%%%%%%%%%%%%%%%%%%%%%%%%%%%%
%\input{biblio}

\end{document}